\documentclass[fleqn,preprint,3p,a4paper]{elsarticle}

\usepackage{amssymb}
\usepackage{amsmath}
\usepackage{amsthm}
\usepackage{dcolumn}
\usepackage{endnotes}
\usepackage{tabularx}
\usepackage[matrix,arrow]{xy}
\usepackage{wasysym}

\newtheorem{theorem}{Theorem}[section]
\newtheorem{proposition}[theorem]{Proposition}
\newtheorem{lemma}[theorem]{Lemma}
\newtheorem{corollary}[theorem]{Corollary}

\theoremstyle{definition}
\newtheorem{definition}[theorem]{Definition}
\newtheorem{example}[theorem]{Example}
\newtheorem{remark}[theorem]{Remark}
\newtheorem{question}[theorem]{Question}

\newcommand{\ir}{{\mathsf{Irr}}}

\newcommand{\cl}{{\rm cl}}
\newcommand{\ii}{{\rm int}}

\newcommand{\ua}{\mathord{\uparrow}}
\newcommand{\da}{\mathord{\downarrow}}

\newcommand{\mk}{\mathord{\mathsf{K}}}
\newcommand{\wdd}{\mathord{\mathsf{WD}}}

\newcommand{\dc}{\mathord{\mathsf{DC}}}

\journal{Topology and its Applications}

\begin{document}

\begin{frontmatter}



\title{Some open problems on well-filtered spaces and sober spaces \tnoteref{t1}}
\tnotetext[t1]{This research was supported by the National Natural Science Foundation of China (11661057); the Natural Science Foundation of Jiangxi Province, China (20192ACBL20045); and NIE ACRF (RI 3/16 ZDS), Singapore}

\author[X. Xu]{Xiaoquan Xu\corref{mycorrespondingauthor}}
\cortext[mycorrespondingauthor]{Corresponding author}
\ead{xiqxu2002@163.com}
\address[X. Xu]{School of Mathematics and Statistics,
Minnan Normal University, Zhangzhou 363000, China}
\author[D. Zhao]{Dongsheng Zhao}
\address[D. Zhao]{Mathematics and Mathematics Education,
National Institute of Education, \\
Nanyang Technological University,
1 Nanyang Walk, Singapore 637616}
\ead{dongsheng.zhao@nie.edu.sg}

\begin{abstract}
In the past few years, the research on sober spaces and well-filtered spaces has got some breakthrough progress. In this paper, we shall present a brief summarising survey on some of such development. Furthermore, we shall pose and illustrate  some open problems on well-filtered spaces and sober spaces.

\end{abstract}

\begin{keyword}
Well-filtered space; Sober space; Strong $d$-space; Co-sober space; Reflection; Smyth power space.

\MSC 54D99; 54B30; 54B20; 06B35

\end{keyword}




\end{frontmatter}


\section{Introduction}

Sobriety is probably the most important and useful property of non-Hausdorff topological spaces. It has been used in the characterizations of spectral spaces and $T_0$ spaces that are determined by their open set lattices. With the development of domain theory, another two properties also emerged as the very useful and important properties for non-Hausdorff topology theory: $d$-spaces and well-filtered spaces (see [1, 11-17, 19-37, 41-43, 45-60]). In the past few years, some remarkable progresses have been achieved in understanding such structures. In the current paper, we shall make a brief survey on some of these progresses, which concern the topics on sobriety of dcpos, existence of well-filterification, finer links between sobriety, coherency, locally compactness, Lawson compactness and well-filteredness.

After reviewing each part, we shall list some major open problems. These problems are not all new, some of them were already posed by other authors. The solutions of such problems, we believe, will further deepen our understanding of the relevant structures. Due to our information limitation, we are not able to include all recent research work on such topics in this paper.

\section{Preliminary}

We now recall some basic concepts and notations that will be used in the paper. For further details, we refer the reader to \cite{Engelking, redbook, Jean-2013}.

Let $P$ be a poset. For any $A\subseteq P$, let
$\mathord{\downarrow}A=\{x\in P: x\leq  a \mbox{ for some }
a\in A\}$ and $\mathord{\uparrow}A=\{x\in P: x\geq  a \mbox{
	for some } a\in A\}$. For each  $x\in P$, we write
$\mathord{\downarrow}x$ for $\mathord{\downarrow}\{x\}$ and
$\mathord{\uparrow}x$ for $\mathord{\uparrow}\{x\}$. A subset $A$ of $P$
is called a \emph{lower set} (resp., an \emph{upper set}) if
$A=\da A$ (resp., $A=\ua A$). Define $P^{(<\omega)}=\{F\subseteq P : F \mbox{~is a nonempty finite set}\}$ and $\mathbf{Fin} ~P=\{\ua F : F\in P^{(<\omega)}\}$. For a nonempty subset $A$ of $P$, define $\mathrm{max}(A)=\{a\in A : a \mbox{~ is a maximal element of~} A\}$ and $\mathrm{min}(A)=\{a\in A : a \mbox{~ is a minimal element of~} A\}$. The symbol $\mathbb{N}$ will denote the poset of all natural numbers with the usual order.

A nonempty subset $D$ of a poset $P$ is called \emph{directed} if every two
elements in $D$ have an upper bound in $D$. The set of all directed sets of $P$ is denoted by $\mathcal D(P)$. A poset $P$ is called a
\emph{directed complete poset}, or \emph{dcpo} for short, if  $\bigvee D$ exists in $P$ for every
$D\in \mathcal D(P)$. A subset $I\subseteq P$ is called an \emph{ideal} if $I$ is a directed lower subset of $P$. Let $\mathrm{Id} (P)$ denote the poset of all ideals of $P$ with the set inclusion order. Dually, we define the \emph{filters} and denote the poset of all filters of $P$ by $\mathrm{Filt}(P)$. The \emph{upper topology} on $P$, generated
by $\{P\setminus \da x : x\in P\}$ as a subbase, is denoted by $\upsilon (P)$. A subset $U$ of $P$ is \emph{Scott open} if
(i) $U=\mathord{\uparrow}U$ and (ii) for any directed subset $D$ for
which $\bigvee D$ exists, $\bigvee D\in U$ implies $D\cap
U\neq\emptyset$. All Scott open subsets of $P$ form a topology, called the \emph{Scott topology} on $P$ and
denoted by $\sigma(P)$. The space $\Sigma~\!\! P=(P,\sigma(P))$ is called the
\emph{Scott space} of $P$. The upper sets of $P$ form the (\emph{upper}) \emph{Alexandroff topology} $\alpha (P)$.

For a $T_0$ space $X$, the \emph{specialization order} $\le_X$ on $X$ is defined by  $x\leq_X y$ if{}f $x\in \overline{\{y\}}$). In the following, when a $T_0$ space is considered as a poset, the order shall means the specialization order provided a different one is specified. Let $\mathcal O(X)$ (resp., $\Gamma (X)$) be the set of all open subsets (resp., closed subsets) of space $X$. Define $\mathcal S_c(X)=\{\overline{{\{x\}}} : x\in X\}$ and $\mathcal D_c(X)=\{\overline{D} : D\in \mathcal D(X)\}$, where $\overline{D}$ is the closure of set $D$. A space $X$ is called a $d$-\emph{space} (or \emph{monotone convergence space}) if $X$ (with the specialization order) is a dcpo
 and $\mathcal O(X) \subseteq \sigma(X)$ (cf. \cite{redbook, Wyler}). For any dcpo $P$, $\Sigma~\!\! P$ is clearly a $d$-space. The category of all $d$-spaces and continuous mappings is denoted by $\mathbf{Top}_d$. For two spaces $X$ and $Y$, we use the symbol $X\cong Y$ to denote that $X$ and $Y$ are homeomorphic.

\begin{lemma}\label{d-space max point}
Let $X$ be a $d$-space. Then for  any nonempty closed subset $A$ of $X$,  $A=\da \mathrm{max}(A)$, and hence $\mathrm{max}(A)\neq\emptyset$.
\end{lemma}
\begin{proof}
 For $x\in A$, by Zorn's Lemma there is a maximal chain $C_x$ in $A$ with $x\in C_x$. Since $X$ is a $d$-space, $c_x=\bigvee C_x$ exists and $c_x\in A$. By the maximality of $C_x$, we have $c_x\in \mathrm{max}(A)$ and $x\leq c_x$. Therefore, $A\subseteq \da \mathrm{max}(A)\subseteq \da A=A$, and hence $A=\da \mathrm{max}(A)$.
\end{proof}

\begin{proposition}\label{d-space function} (\cite{redbook})
	For any $T_0$ space $X$ and  $d$-space $Y$, the function space $\mathbf{Top}_0(X, Y)$ equipped with the pointwise convergence topology is a $d$-space.
\end{proposition}

A nonempty subset $A$ of a $T_0$ space $X$ is called \emph{irreducible} if for any $\{F_1, F_2\}\subseteq \Gamma (X)$, $A \subseteq F_1\cup F_2$ implies $A \subseteq F_1$ or $A \subseteq  F_2$.  We denote by $\ir(X)$ (resp., $\ir_c(X)$) the set of all irreducible (resp., irreducible closed) subsets of $X$. Clearly, every subset of $X$ that is directed under $\leq_X$ is irreducible. A topological space $Y$ is called \emph{sober}, if for any  $F\in\ir_c(Y)$, there is a unique point $a\in Y$ such that $F=\overline{\{a\}}$, the closure of $\{a\}$. The category of all sober spaces and continuous mappings is denoted by $\mathbf{Sob}$. Let $\mathrm{OFilt(\mathcal O(X))}=\sigma (\mathcal O(X))\bigcap \mathrm{Filt}(\mathcal O(X))$. The members of $\mathrm{OFilt(\mathcal O(X))}$ are called open filters of $X$.

For any space $X$, let $\mk(X)$ be the set of all nonempty compact subsets of $X$. The order on $\mk(X)$ is usually taken as the reverse inclusion order.

For each $K\in \mk (X)$, let $\Phi (K)=\{U\in \mathcal O(X) : K\subseteq U\}$. Then $\Phi (K)\in \mathrm{OFilt(\mathcal O(X))}$ and $K=\bigcap \Phi (K)$. Obviously, $\Phi : \mk (X) \longrightarrow \mathrm{OFilt(\mathcal O(X))}, K\mapsto \Phi (K)$, is an order embedding.

For a $T_0$ space $X$, $\mathcal G\subseteq 2^{X}$ and $W\subseteq X$, let $\Diamond_{\mathcal G} W=\{G\in \mathcal G : G\bigcap W\neq\emptyset\}$ and $\Box_{\mathcal G} W=\{G\in \mathcal G : G\subseteq  A\}$. The symbols $\Diamond_{\mathcal G} W$ and $\Box_{\mathcal G} W$ will be simply written as $\Diamond A$ and $\Box A$ respectively if no ambiguity occur. The \emph{lower Vietoris topology} on $\mathcal{G}$ is the topology that has $\{\Diamond U : U\in \mathcal O(X)\}$ as a subbase, and the resulting space is denoted by $P_H(\mathcal{G})$. If $\mathcal{G}\subseteq \ir (X)$, then $\{\Diamond_{\mathcal{G}} U : U\in \mathcal O(X)\}$ is a topology on $\mathcal{G}$. The space $P_H(\Gamma(X)\setminus \{\emptyset\})$ is called the \emph{Hoare power space} or \emph{lower space} of $X$ and is denoted by $P_H(X)$ for short (cf. \cite{Schalk}). Clearly, $P_H(X)=(\Gamma(X)\setminus \{\emptyset\}, \upsilon(\Gamma(X)\setminus \{\emptyset\}))$. So $P_H(X)$ is always sober (see \cite[Corollary 4.10]{ZhaoHo} or \cite[Proposition 2.9]{xu-shen-xi-zhao1}).

\begin{remark} \label{eta continuous} Let $X$ be a $T_0$ space.
\begin{enumerate}[\rm (1)]
	\item If $\mathcal{S}_c(X)\subseteq \mathcal{G}$, then the specialization order on $P_H(\mathcal{G})$ is the set inclusion order, and the \emph{canonical mapping} $\eta_{X}: X\longrightarrow P_H(\mathcal{G})$, given by $\eta_X(x)=\overline {\{x\}}$, is an order and topological embedding (cf. \cite{redbook, Jean-2013, Schalk}).
    \item The space $X^s=P_H(\ir_c(X))$ with the canonical mapping $\eta_{X}: X\longrightarrow X^s$ is the \emph{sobrification} of $X$ (cf. \cite{redbook, Jean-2013}).
\end{enumerate}
\end{remark}

A subset $A$ of a $T_0$ space $X$ is called \emph{saturated} if $A$ equals the intersection of all open sets containing it (equivalently, $A$ is an upper set with respect to the specialization order). We shall use $\mathord{\mathsf{K}}(X)$ to
denote the set of all nonempty compact saturated subsets of $X$ equipped with the \emph{Smyth preorder}: for $K_1,K_2\in \mathord{\mathsf{K}}(X)$, $K_1\sqsubseteq K_2$ if{}f $K_2\subseteq K_1$. For $U\in \mathcal O(X)$, let $\Box U=\{K\in \mk (X) : K\subseteq  U\}$. The \emph{upper Vietoris topology} on $\mk (X)$ is the topology generated by $\{\Box U : U\in \mathcal O(X)\}$ as a base, and the resulting space is called the \emph{Smyth power space} or \emph{upper space} of $X$ and is denoted by $P_S(X)$ (cf. \cite{Heckmann, Klause-Heckmann, Schalk}).

For a nonempty subset $C$ of a $T_0$ space $X$, it is easy to see that $C$ is compact if{}f $\ua C\in \mk (X)$. Furthermore, we have the following useful result (see, e.g., \cite[pp.\! 2068]{E_2009}).

\begin{lemma}\label{COMPminimalset}  Let $X$ be a $T_0$ space and $C\in \mk (X)$. Then $C=\ua \mathrm{min}(C)$ and  $\mathrm{min}(C)$ is compact.
\end{lemma}

A space $X$ is called \emph{well-filtered} if it is $T_0$, and for any open set $U$ and any $\mathcal{K}\in\mathcal D(\mk (X))$, $\bigcap\mathcal{K}{\subseteq} U$ implies $K{\subseteq} U$ for some $K{\in}\mathcal{K}$. The category of all well-filtered spaces and  continuous mappings is denoted by $\mathbf{Top}_w$.

\begin{remark} \label{xi continuous} Let $X$ be a $T_0$ space. Then
\begin{enumerate}[\rm (1)]
	\item the specialization order on $P_S(X)$ is the Smyth order, that is, $\leq_{P_S(X)}=\sqsubseteq$;
    \item the \emph{canonical mapping} $\xi_X: X\longrightarrow P_S(X)$, $x\mapsto\ua x$, is an order and topological embedding (cf. \cite{Heckmann, Klause-Heckmann, Schalk}).
\end{enumerate}
\end{remark}

The Smyth power space construction defines  a covariant functor. More precisely, we have the following.

\begin{lemma}\label{Ps functor} (\cite{xu2020}) $P_S : \mathbf{Top}_0 \longrightarrow \mathbf{Top}_0$ is a covariant functor, where for any $f : X \longrightarrow Y$ in $\mathbf{Top}_0$, $P_S(f) : P_S(X) \longrightarrow P_S(Y)$ is defined by $P_S(f)(K)=\ua f(K)$ for all $ K\in\mk(X)$.
\end{lemma}

As in \cite{E_20182}, a topological space $X$ is \emph{locally hypercompact} if for each $x\in X$ and each open neighborhood $U$ of $x$, there is  $\ua F\in \mathbf{Fin}~X$ such that $x\in\ii\,\ua F\subseteq\ua F\subseteq U$. A space $X$ is called a $C$-\emph{space} if for each $x\in X$ and each open neighborhood $U$ of $x$, there is $u\in X$ such that $x\in\ii\,\ua u\subseteq\ua u\subseteq U$). A set $K\subseteq X$ is called \emph{supercompact} if for
any family $\{U_i : i\in I\}\subseteq \mathcal O(X)$, $K\subseteq \bigcup_{i\in I} U_i$  implies $K\subseteq U$ for some $i\in I$. It is easy to see that the supercompact saturated sets of $X$ are exactly the sets $\ua x$ with $x \in X$ (see \cite[Fact 2.2]{Klause-Heckmann}). It is well-known that $X$ is a $C$-space if{}f $\mathcal O(X)$ is a \emph{completely distributive} lattice (cf. \cite{E_2009}). A space $X$ is called \emph{core compact} if $(\mathcal O(X), \subseteq)$ is a \emph{continuous lattice} (cf. \cite{redbook}).


For a full subcategory $\mathbf{K}$ of $\mathbf{Top}_0$, the objects of $\mathbf{K}$ will be called $\mathbf{K}$-spaces. In \cite{Keimel-Lawson}, Keimel and Lawson proposed the following properties:

($\mathrm{K}_1$) Homeomorphic copies of $\mathbf{K}$-spaces are $\mathbf{K}$-spaces.

($\mathrm{K}_2$) All sober spaces are $\mathbf{K}$-spaces or, equivalently, $\mathbf{Sob}\subseteq \mathbf{K}$.

($\mathrm{K}_3$) In a sober space S, the intersection of any family of $\mathbf{K}$-subspaces is a $\mathbf{K}$-space.

($\mathrm{K}_4$) Continuous maps $f : S \longrightarrow T$ between sober spaces $S$ and $T$ are $\mathbf{K}$-continuous, that is, for every $\mathbf{K}$-subspace $K$ of $T$ , the inverse image $f^{-1}(K)$ is a $\mathbf{K}$-subspace of $S$.

In what follows, $\mathbf{K}$ always refers to a full subcategory $\mathbf{Top}_0$ containing $\mathbf{Sob}$, that is, $\mathbf{K}$ has ($\mathrm{K}_2$). $\mathbf{K}$ is said to be \emph{closed with respect to homeomorphisms} if $\mathbf{K}$ has ($\mathrm{K}_1$). We call $\mathbf{K}$ a \emph{Keimel-Lawson category} if it satisfies ($\mathrm{K}_1$)-($\mathrm{K}_4$).

\section{Well-filtered spaces and sober spaces}

 It is well-known that every sober space is well-filtered (see \cite{Hofmann-Mislove}) and every well-filtered space is a $d$-space (cf. \cite{Xi-Lawson-2017, xu-shen-xi-zhao1}). The Scott space of every continuous
dcpo is sober (see \cite{redbook}). Furthermore, the Scott space of every quasicontinuous
domain is sober (see \cite{quasicont}). Johnstone
\cite{johnstone-81} constructed the first dcpo whose Scott
space is non-sober. Soon after, Isbell \cite{isbell}
gave a complete lattice whose Scott space is non-sober. The general problem in this line is whether each object in
a classic class of lattices has a sober Scott space.

In 1992, Heckmann \cite{Heckmann} asked the following question:

\begin{center} Is every well-filtered dcpo sober in its Scott-topology?
\end{center}

The answer is negative. Kou \cite{Kou} constructed the
first dcpo whose Scott space is well-filtered but
non-sober (see \cite{jia-2018, Xi-Lawson-2017, zhao-xi-chen} for other different counterexamples).

In \cite{Xi-Lawson-2017}, Xi and Lawson proved a sufficient condition for a $T_0$ space to be well-filtered.

\begin{theorem}\label{xi-lawsonWF} (\cite{Xi-Lawson-2017})
	Let $X$ be a $d$-space with the property that $\downarrow\!\! (A\cap K)$ is closed for any
$A\in \Gamma (X)$ and $K\in \mk (X)$.  Then $X$ is well-filtered.
\end{theorem}

\begin{corollary}\label{Xi-Lawson result 1} (\cite{Xi-Lawson-2017})
If a dcpo $P$ is Lawson-compact (in particular, bounded-complete), then $\Sigma P$ is well-filtered.
\end{corollary}

It follows from Theorem \ref{xi-lawsonWF} or Corollary \ref{Xi-Lawson result 1} that Isbell's non-sober complete lattice is well-filtered. Note that Johnstone's dcpo is countable and Isbell's complete lattice is not countable. In 2019, in a talk \cite{Jung19} given in National Institute of Education, Singapore, Achim Jung posed the following problem which is still open.

\begin{question}\label{Jung problem1} (\cite{Jung19})  Is there a countable complete lattice that has a non-sober Scott topology?
\end{question}

Isbell's non-sober complete lattice is not distributive. Thus in 1994, Jung also asked
whether there is a distributive complete lattice whose Scott
space is non-sober (see \cite[Exercises 7.3.19-6]{AJ94} or \cite{Jung19}).

In \cite{xuxizhao}, using Isbell's lattice, we give a positive answer
to Jung's above problem.

\begin{theorem}\label{spacial frame Scott is non-sober} (\cite{xuxizhao}) There is a spatial frame whose
Scott space is non-sober.
\end{theorem}

Note that the dual of a spatial frame need not be a frame. Thus a question naturally arising is the following.

\begin{question}\label{lattice of closed sets Scott sober question} Is there a complete lattice with enough co-primes (i.e.,
isomorphic to the lattice of all closed subsets of a topological space) that
has a non-sober Scott topology?
\end{question}

In \cite{EG85}, it is shown that for the complete Boolean algebra $B$ of all regular open subsets of the reals line, the Scott space $\Sigma B$ is not a topological join-semilattice (and hence the Scott topology $\sigma (B\times B)$ is properly larger than the product topology $\sigma (B)\times \sigma (B)$). It is natural to wonder whether $\Sigma B$ is sober. Thus Ern\'e \cite{ELX19} asked the following.

\begin{question}\label{regur open sets Scott sober question} (\cite{ELX19})  Let $B$ be the complete Boolean algebra of all regular open
subsets of the reals line. Is the Scott space $\Sigma B$ sober?
\end{question}

If the answer would be in the affirmative, we would have an example of a sober complete Scott space that fails to be a topological join-semilattice; if it would be in the negative, we would have a second, more natural example of a non-sober Scott space of a complete lattice (which is really a complete Boolean algebra).

Another such type of problem is the following.

\begin{question}\label{cBa Scott sober question} Is there a complete Boolean algebra that has a non-sober
Scott topology?
\end{question}

One of the most important result on sober spaces is the Hofmann-Mislove Theorem (see \cite[Theorem 2.16]{Hofmann-Mislove} or \cite[Theorem II-1.20 and Theorem II-1.21]{redbook}).

\begin{theorem}\label{Hofmann-Mislove theorem} \emph{(Hofmann-Mislove Theorem)} For a $T_0$ space $X$, the following conditions are equivalent:
\begin{enumerate}[\rm (1)]
            \item $X$ is a sober space.
            \item  For any $\mathcal F\in \mathrm{OFilt}(\mathcal O(X))$, there is a $K\in \mk (X)$ such that $\mathcal F=\Phi (K)$.
            \item  For any $\mathcal F\in \mathrm{OFilt}(\mathcal O(X))$, $\mathcal F=\Phi (\bigcap \mathcal F)$.
\end{enumerate}
\end{theorem}

By Hofmann-Mislove Theorem, $\Phi : \mk (X) \longrightarrow \mathrm{OFilt(\mathcal O(X))}$ is an order isomorphism if and only if $X$ is sober.

For locally compact well-filtered spaces, we have the following well-known result (see, e.g., \cite{redbook, Jean-2013, Kou}).

\begin{theorem}\label{SoberLC=CoreC}  For a $T_0$ space $X$, the following conditions are equivalent:
\begin{enumerate}[\rm (1)]
	\item $X$ is locally compact and sober.
	\item $X$ is locally compact and well-filtered.
	\item $X$ is core compact and sober.
\end{enumerate}
\end{theorem}

In \cite{jia-2018}, Jia Xiaodong asks the following question.

\begin{center} Is every core compact and well-filtered space sober?
\end{center}

This problem was first solved by Lawson and Xi \cite{Lawson-Xi}, and later answered by Xu, Shen, Xi and Zhao \cite{xu-shen-xi-zhao1, xu-shen-xi-zhao2} using a different method.

\begin{theorem}\label{core compact WF is sober} (\cite{Lawson-Xi, xu-shen-xi-zhao1, xu-shen-xi-zhao2})
 Every core compact well-filtered space is sober.
\end{theorem}

In addition, we have proved the following.

\begin{theorem}\label{WF CI is sober} (\cite{xu-shen-xi-zhao2})
 Every first countable well-filtered $T_0$ space is sober.
\end{theorem}

By Theorem \ref{core compact WF is sober}, Theorem \ref{SoberLC=CoreC} can be strengthened to the following one.

\begin{theorem}\label{SoberLC=CoreCNew}  For a $T_0$ space $X$, the following conditions are equivalent:
\begin{enumerate}[\rm (1)]
	\item $X$ is locally compact and sober.
	\item $X$ is locally compact and well-filtered.
	\item $X$ is core compact and sober.
    \item $X$ is core compact and well-filtered.
\end{enumerate}
\end{theorem}

Jia \cite{jia-2018} also asked the following question.

\begin{question}\label{Scott core compact is sober question} (\cite{jia-2018})
If $L$ is a meet-continuous dcpo and $\Sigma P$ is core compact, is $\Sigma P$ sober?
\end{question}

The following related question arises naturally.

\begin{question}\label{Scott locally compact is sober question}
If the Scott topology on a dcpo $P$ is locally compact, is $\Sigma L$ sober?
\end{question}

\section{Scott sober dcpos and Scott well-filtered dcpos}

In \cite{EL08}, the sober posets are defined and studied. A poset $P$ is said to be \emph{sober} if there
exists a sober topology $\tau$ on $P$ which is compatible with the original order of $P$ (i.e., $\cl_\tau \{x\} =\da x$ for each $x\in P$).

In a similar manner, we introduce the well-filtered posets.

\begin{definition}\label{WF poset} A poset $P$ is said to be \emph{well}-\emph{filtered} if there
exists a well-filtered topology $\tau$ on $P$ which is compatible with the original order of $P$.
\end{definition}

Clearly, sober posets are well-filtered and well-filtered posets are dcpos.

\begin{example}\label{dcpo is not well-filtered} (Johnstone's dcpo)  Let $\mathbb{J}=\mathbb{N}\times (\mathbb{N}\cup \{\infty\})$ with ordering defined by $(j, k)\leq (m, n)$ if{}f $j = m$ and $k \leq n$, or $n =\infty$ and $k\leq m$ (see Figure 1).

\begin{figure}[ht]
	\centering
	\includegraphics[height=4.5cm,width=4.5cm]{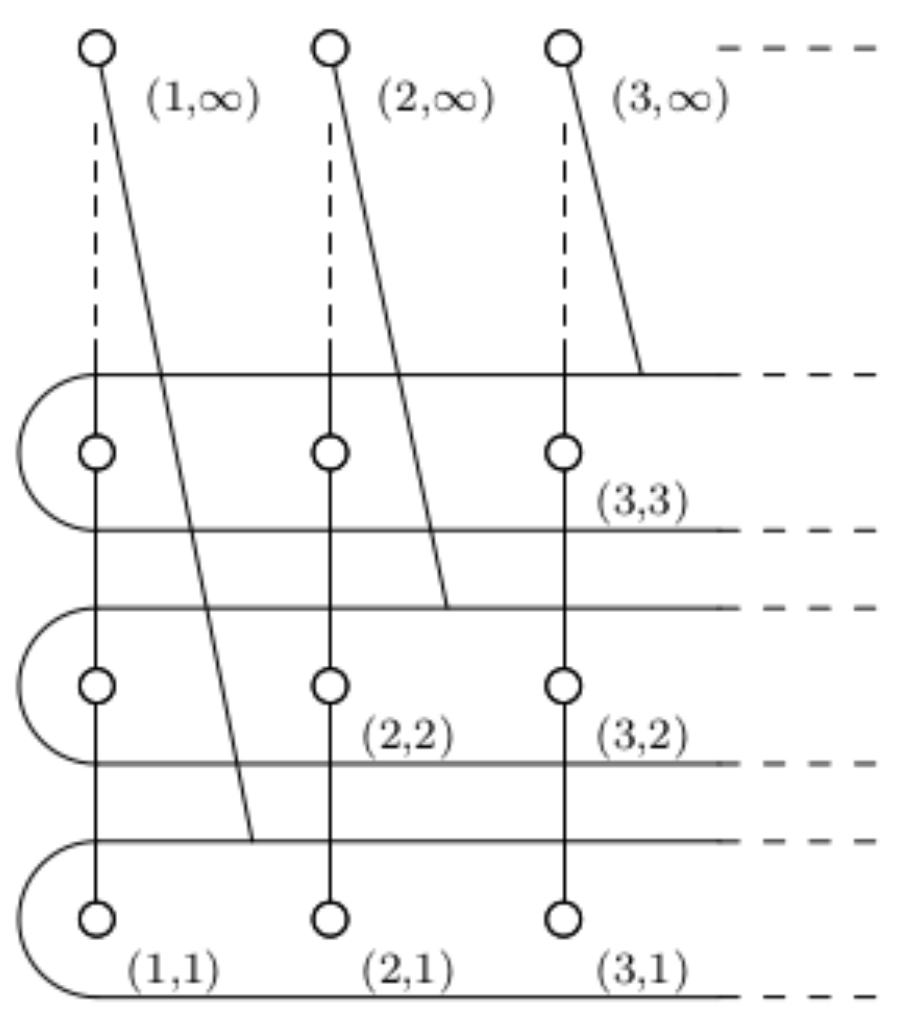}
	\caption{Johnstone's dcpo $\mathbb{J}$}
\end{figure}

$\mathbb{J}$ is a well-known dcpo constructed by Johnstone in \cite{johnstone-81}. Now we show that $\mathbb{J}$ is not well-filtered. Suppose, on the contrary, that there
exists a well-filtered topology $\tau$ on $\mathbb{J}$ which is compatible with the original order of $\mathbb{J}$. Clearly, $\bigcap_{n\in \mathbb{N}}(\ua (1, n)\cap \ua (2,1))=\emptyset$, and $\ua (1, n)\cap \ua (2,1)=\{(m, \infty): n\leq m\}$ is a compact saturated subset in $\Sigma \mathbb{J}$, and hence a compact saturated subset in $(\mathbb{J}, \tau)$ since $\tau\subseteq \sigma (\mathbb{J})$. By the well-filteredness of $(\mathbb{J}, \tau)$, $\ua (1, n_0)\cap \ua (2,1)=\emptyset$ for some $n_0\in \mathbb{N}$, a contradiction. Thus $\mathbb{J}$ is not well-filtered. In particular, $\Sigma~\!\mathbb{J}$ is not well-filtered (see \cite[Exercise 8.3.9]{Jean-2013}).
\end{example}

It is still not known whether there is a well-filtered dcpo that is  not
sober. For Isblle's lattice $L$ constructed in \cite{isbell}, $\Sigma~\!\! L$ is non-sober, and by Corollary \ref{Xi-Lawson result 1}, $\Sigma~\!\! L$ is well-filtered, and hence $L$ is well-filtered. But we do not know whether $L$ is sober.

\begin{question}\label{WF poset question} Characterize the well-filtered dcpos.
\end{question}

In order to emphasize the Scott topology, we introduce the following notions.

\begin{definition}\label{Scott dcpo WF dcpo} Let $P$ be a dcpo.
\begin{enumerate}[\rm (1)]
\item $P$ is said to be \emph{Scott sober} if $\Sigma P$ is sober.
\item $P$ is said to be \emph{Scott well-filtered} if $\Sigma P$ is well-filtered.
\end{enumerate}
\end{definition}

Regarding these, we have  the following five problems.

\begin{question}\label{WF dcpo charac question} Find an order characterization for
Scott well-filtered dcpos (cf. the first problem below \cite[Exercise VII-1.14]{redbook}).
\end{question}

\begin{question}\label{Scott sober dcpo prod question} Let $P, Q$ be Scott sober dcpos. Must the product $P\times Q$ be Scott sober?
\end{question}

\begin{question}\label{Scott well-filtered dcpo prod charac question} Let $P, Q$ be well-filtered dcpos. Must the product $P\times Q$ be Scott well-filtered?
\end{question}

If $P$ is a domain, then one can prove that for any sober dcpo $Q$, the product $P\times Q$ is sober. Thus it is natural to wonder which other dcpos also own this property.

\begin{question}\label{Scott well-filtered dcpo prod question} Characterize those Scott sober dcpos $P$ for which the products $P\times Q$ is Scott sober for every Scott sober dcpo $Q$.
\end{question}

\begin{question}\label{Scott sober dcpo prod charac question} Characterize those Scott well-filtered dcpos $P$ for which the products $P\times Q$ is Scott well-filtered for every Scott well-filtered dcpo $Q$.
\end{question}

Another problem concerning the sober dcpos is the following one.

\begin{question}\label{equivalent-weaker sober property}
Is there a topological property $p$ which is weaker than the sobriety for general $T_0$ topological spaces but equivalent to the sobriety for Scott spaces of dcpos?
\end{question}

\section{Rudin sets and well-filtered determined sets}

Rudin's Lemma is a very useful tool in non-Hausdorff topology and plays a crucial role in domain theory (see [1, 12, 13, 19-24, 30, 41, 42, 49-53, 55]). Marry Ellen Rudin \cite{Rudin} proved her lemma by transfinite methods, using the Axiom of Choice.
Heckman and Keimel \cite{Klause-Heckmann} obtained the following topological variant of Rudin's Lemma.

\begin{lemma}\label{t Rudin} \emph{(Topological Rudin's Lemma)} Let $X$ be a topological space and $\mathcal{A}$ an
irreducible subset of the Smyth power space $P_S(X)$. Then every closed set $C {\subseteq} X$  that
meets all members of $\mathcal{A}$ contains an minimal irreducible closed subset $A$ that still meets all
members of $\mathcal{A}$.
\end{lemma}

Applying Lemma \ref{t Rudin} to the Alexandroff topology on a poset $P$, one obtains  the original Rudin's Lemma.

\begin{corollary}\label{rudin} \emph{(Rudin's Lemma)} Let $P$ be a poset, $C$ a nonempty lower subset of $P$ and $\mathcal F\in \mathbf{Fin}~P$ a filtered family with $\mathcal F\subseteq\Diamond C$. Then there exists a directed subset $D$ of $C$ such that $\mathcal F\subseteq \Diamond\da D$.
\end{corollary}

\begin{definition}\label{DCspace} (\cite{xu-shen-xi-zhao1})
	A $T_0$ space $X$ is called a \emph{directed closure space}, $\mathsf{DC}$ \emph{space} for short, if $\ir_c(X)=\mathcal{D}_c(X)$, that is, for each $A\in \ir_c(X)$, there exists a directed subset of $X$ such that $A=\overline{D}$.
\end{definition}

It is easy to verify that closed subspaces, retracts and products of $\dc$ spaces are again  $\dc$ spaces (see \cite{xu-shen-xi-zhao1}).

For a $T_0$ space $X$ and $\mathcal{K}\subseteq \mathord{\mathsf{K}}(X)$, let $M(\mathcal{K})=\{A\in \Gamma (X) : K\bigcap A\neq\emptyset \mbox{~for all~} K\in \mathcal{K}\}$ (that is, $\mathcal A\subseteq \Diamond A$) and $m(\mathcal{K})=\{A\in \Gamma (X) : A \mbox{~is a minimal menber of~} M(\mathcal{K})\}$.

Based on the topological Rudin's Lemma, we introduce a new type of spaces --- Rudin spaces (see \cite{Shenchon, xu-shen-xi-zhao1}).

\begin{definition}\label{Rudin space WD space}
	Let $X$ be a $T_0$ space and $A$ a nonempty subset of $X$.
\begin{enumerate}[\rm (1)]
\item $A$ is said to have the \emph{Rudin property}, if there exists a filtered family $\mathcal K\subseteq \mathord{\mathsf{K}}(X)$ such that $\overline{A}\in m(\mathcal K)$ (that is, $\overline{A}$ is a minimal closed set that intersects all members of $\mathcal K$). Let $\mathsf{RD}(X)=\{A\in \Gamma (X) : A\mbox{~has Rudin property}\}$. The sets in $\mathsf{RD}(X)$ will also be called \emph{Rudin sets}.
\item $X$ is called a \emph{Rudin space}, $\mathsf{RD}$ \emph{space} for short, if $\ir_c(X)=\mathsf{RD}(X)$, that is, every irreducible closed set of $X$ is a Rudin set.
\end{enumerate}
\end{definition}

The Rudin property is called the \emph{compactly filtered property} in \cite{Shenchon}. Here, in order to emphasize its root from (topological) Rudin's Lemma, we call such a property the Rudin property.

Now we define another type of spaces --- $\mathbf{K}$-determined spaces.

\begin{definition}\label{KD subset} (\cite{xu20}) Let $\mathbf{K}$ be a full category of $\mathbf{Top}_0$ containing $\mathbf{Sob}$ and $X$ a $T_0$ space.
	 A subset $A$ of $X$ is called a $\mathbf{K}$-\emph{determined set}, provided for any continuous mapping $ f:X\longrightarrow Y$
to a $\mathbf{K}$-space $Y$, there exists a unique $y_A\in Y$ such that $\overline{f(A)}=\overline{\{y_A\}}$.
Denote by $\mathbf{K}(X)$ the set of all closed $\mathbf{K}$-determined sets of $X$. The space $X$ is said to be a $\mathbf{K}$-\emph{determined} space, if $\ir_c(X)=\mathbf{K}(X)$ or, equivalently, all irreducible closed sets of $X$ are $\mathbf{K}$-sets.
\end{definition}

Clearly, a subset $A$ of a space $X$ is  $\mathbf{K}$-determined  if{}f $\overline{A}$ is a $\mathbf{K}$-determined set. For simplicity, let $d(X)=\mathbf{Top}_d(X)$ and $\mathsf{WD}(X)=\mathsf{Top}_w(X)$. The sets in $\mathsf{WD}(X)$ are called $\mathsf{WD}$ sets. The space $X$ is called a \emph{well-filtered determined} space, shortly a $\mathsf{WD}$ space, if all irreducible closed subsets of $X$ are $\mathsf{WD}$ sets, that is, $\ir_c(X)=\mathsf{WD} (X)$ (see \cite{xu20, xu-shen-xi-zhao1}).

\begin{lemma}\label{sobd=irr} (\cite{xu20}) Let $\mathbf{K}$ be a full category of $\mathbf{Top}_0$ containing $\mathbf{Sob}$ and $X$ a $T_0$ space. Then $\mathcal S_c(X)\subseteq\mathbf{KD}(X)\subseteq\mathbf{Sob}(X)=\ir_c(X)$.
\end{lemma}

\begin{proposition}\label{DRWIsetrelation} (\cite{xu-shen-xi-zhao1})
	For any $T_0$ space  $X$, $\mathcal S_c(X)\subseteq \mathcal{D}_c(X)\subseteq \mathsf{RD}(X)\subseteq\mathsf{WD}(X)\subseteq\ir_c(X)$.
\end{proposition}

\begin{corollary}\label{SDRWspacerelation} (\cite{xu-shen-xi-zhao1})
	Sober $\Rightarrow$ $\mathsf{DC}$ $\Rightarrow$ $\mathsf{RD}$ $\Rightarrow$ $\mathsf{WD}$.
\end{corollary}

By Lemma \ref{sobd=irr}, sober spaces are $\mathbf{K}$-determined. By Proposition \ref{DRWIsetrelation} we know that the class of Rudin spaces lie between the class of $\mathsf{WD}$ spaces and that of $\dc$ spaces. Also the class of $\dc$ spaces lie between the class of Rudin spaces and that of sober spaces.

In \cite[Example 4.15]{Liu-Li-Wu-2020}, Liu, Li and Wu constructed a $T_0$ space $X$ in which some well-filtered determined sets are not Rudin sets, and hence gave a negative answer to a queston posed by Xu and Zhao in \cite{xuzhao}: Does $\mathsf{RD}(X)=\mathsf{WD}(X)$ hold for every $T_0$ space $X$? It is not difficult to check that the space $X$ is a $\mathsf{WD}$ space but not a Rudin space. Therefore, Example 4.15 in \cite{Liu-Li-Wu-2020} also gave a negative answer to another related question raised by Xu, Shen, Xi and Zhao in \cite{xu-shen-xi-zhao1}: Is every well-filtered determined space a Rudin space?

Using Rudin sets and $\wdd$ sets, we can give some new characterizations of well-filtered spaces and sober spaces.

\begin{theorem}\label{WFwdc} (\cite{xu-shen-xi-zhao1})
	For a $T_0$ space $X$, the following conditions are equivalent:
	\begin{enumerate}[\rm (1)]
		\item $X$ is well-filtered.
		\item $\mathsf{RD}(X)=\mathcal S_c(X)$.
        \item $\wdd (X)=\mathcal S_c(X)$, that is, for each $A\in\wdd(X)$, there exists a unique $x\in X$ such that $A=\overline{\{x\}}$.
	\end{enumerate}
\end{theorem}

\begin{theorem}\label{soberequiv}  (\cite{xu-shen-xi-zhao1}) For a $T_0$ space $X$, the following conditions are equivalent:
	\begin{enumerate}[\rm (1)]
		\item $X$ is sober.
		\item $X$ is a $\mathsf{DC}$ $d$-space.
        \item $X$ is a well-filtered $\mathsf{DC}$ space.
		\item $X$ is a well-filtered Rudin space.
		\item $X$ is a well-filtered $\mathsf{WD}$ space.
	\end{enumerate}
\end{theorem}

\begin{lemma}\label{LHCdirected} (\cite{E_20182})
	Let $X$ be a locally hypercompact $T_0$ space and $A\in\ir(X)$. Then there exists a directed subset $D\subseteq\da A$ such that $\overline{A}=\overline{D}$.
\end{lemma}

Corollary \ref{SDRWspacerelation} and Lemma \ref{LHCdirected} together imply the following.

\begin{corollary}\label{erne1} (\cite{xu-shen-xi-zhao1})
	If $X$ is a locally hypercompact $T_0$ space, then it is a $\mathsf{DC}$ space. Therefore, it is a Rudin space and a $\mathsf{WD}$ space.
\end{corollary}

\begin{theorem}\label{LCrudin} (\cite{xu-shen-xi-zhao1}) Every locally compact $T_0$ space is a Rudin space.
\end{theorem}

Similarly, we have the following result, which positively answers \cite[Problem 4.2]{xuzhao}.

\begin{theorem}\label{Core compt is WD} (\cite{xu-shen-xi-zhao1}) Every core compact $T_0$ space is well-filtered determined.
\end{theorem}

Figure 2 shows some relationships among some types of spaces.

\begin{figure}[ht]
	\centering
	\includegraphics[height=2.2in,width=4.0in]{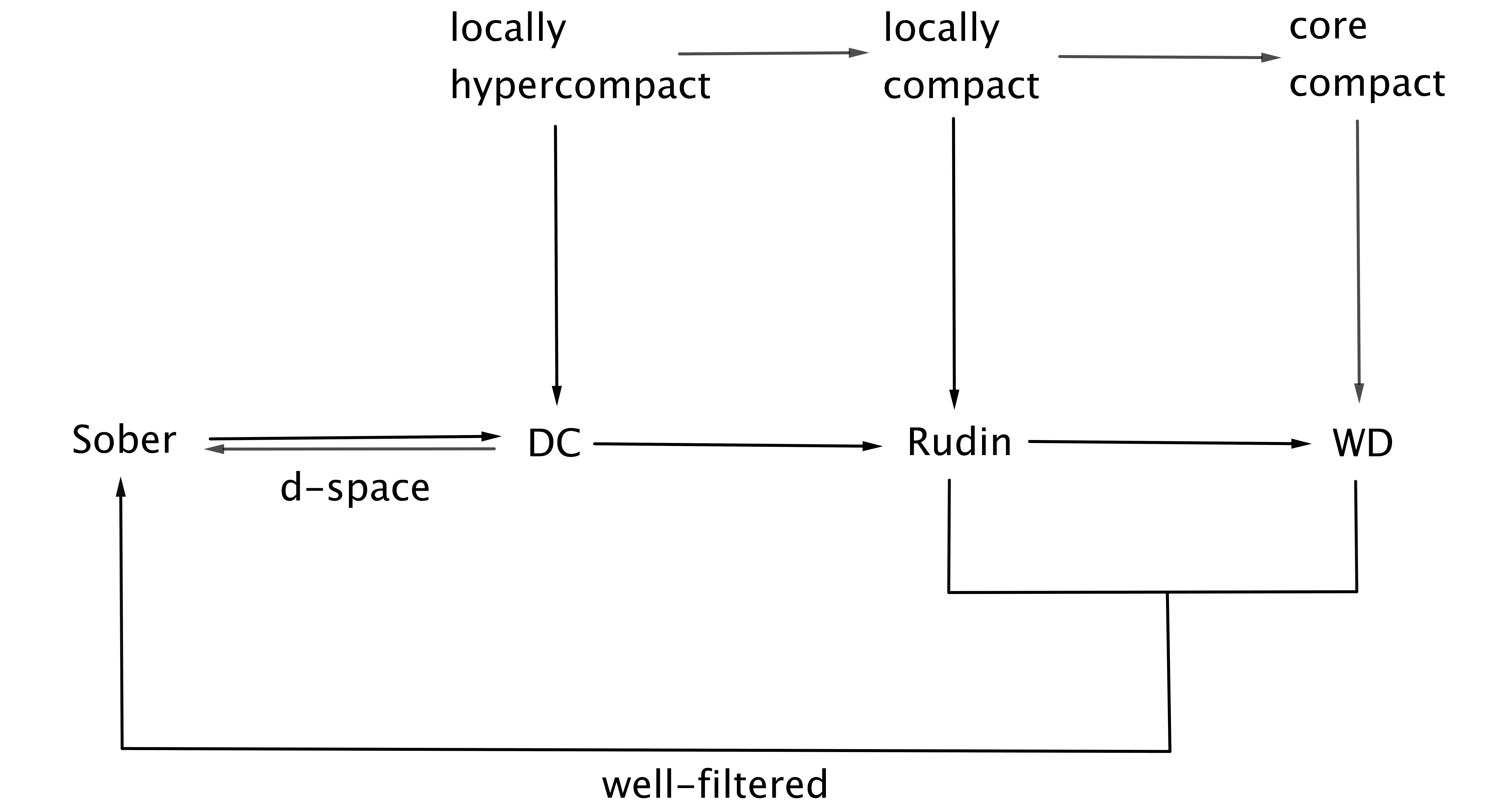}
	\caption{Certain relations among some kinds of spaces}
\end{figure}

\begin{question}\label{core compact is Rudin space question}  Is every core compact $T_0$ space a Rudin space?
\end{question}

In \cite{Shenchon}, it is shown that the closed subspaces and retracts of Rudin spaces are again Rudin spaces.

\begin{lemma}\label{Rudinsetprod} (\cite{Shenchon})
	Let	$X=\prod_{i\in I}X_i$ be the product of a family $\{X_i:i\in I\}$ of $T_0$ spaces and $A\in \ir (X)$. Then the following conditions are equivalent:
\begin{enumerate}[\rm (1)]
	\item $A$ is a Rudin set.
	\item $p_i(A)$ is a Rudin set for each $i\in I$.
\end{enumerate}
\end{lemma}

\begin{theorem}\label{rudinprod} (\cite{xu-shen-xi-zhao1})
	Let $\{X_i:i\in I\}$ be a family of $T_0$ spaces. Then the following two conditions are equivalent:
	\begin{enumerate}[\rm(1)]
		\item The product space $\prod_{i\in I}X_i$ is a Rudin space.
		\item For each $i \in I$, $X_i$ is a Rudin space.
	\end{enumerate}
\end{theorem}

It is proved in \cite{xu-shen-xi-zhao1} that closed subspaces and retracts of well-filtered determined spaces are again well-filtered determined spaces. But we do not know wether a saturated (especially, an open) subspace of a well-filtered determined space is still well-filtered determined.

\begin{lemma}\label{WDsetprod} (\cite{xu-shen-xi-zhao1})
	Let	$\{X_i: 1\leq i\leq n\}$ be a finite family of $T_0$ spaces and $X=\prod\limits_{i=1}^{n}X_i$ the product space. For $A\in\ir (X)$, the following conditions are equivalent:
\begin{enumerate}[\rm (1)]
	\item $A$ is a $\wdd$ set.
	\item $p_i(A)$ is a $\wdd$ set for each $1\leq i\leq n$.
\end{enumerate}
\end{lemma}

\begin{corollary}\label{WDclosedsetprod} (\cite{xu-shen-xi-zhao1})
	Let	$X=\prod\limits_{i=1}^{n}X_i$ be the product of a finite family $\{X_i: 1\leq i\leq n\}$ of $T_0$ spaces. If $A\in\wdd (X)$, then $A=\prod\limits_{i=1}^{n}p_i(X_i)$, and $p_i(A)\in \wdd (X_i)$ for all $1\leq i \leq n$.
\end{corollary}

\begin{corollary}\label{WDspace prod} (\cite{xu-shen-xi-zhao1})
	Let $\{X_i: 1\leq i\leq n\}$ be a finite family of $T_0$ spaces. Then the following two conditions are equivalent:
	\begin{enumerate}[\rm(1)]
		\item The product space $\prod\limits_{i=1}^{n}X_i$ is a well-filtered determined space.
		\item For each $1\leq i \leq n$, $X_i$ is a well-filtered determined space.
	\end{enumerate}
\end{corollary}

\begin{question}\label{K-infiniteset-prodquestion} (\cite{xu20}) Let $\mathbf{K}$ be a full subcategory of $\mathbf{Top}_0$ containing $\mathbf{Sob}$ (in particular, a Keimel-Lawson category) and $X=\prod_{i\in I}X_i$ be  the product space of a family $\{X_i: i\in I\}$ of $T_0$ spaces. If each $A_i\subseteq X_i~(i\in I)$ is $\mathbf{K}$-determined, must the product set $\prod_{i\in I}A_i$ be $\mathbf{K}$-determined?
\end{question}

\begin{question}\label{WDinfinite-prodquestion} (\cite{xu20}) Let $\mathbf{K}$ be a full category of $\mathbf{Top}_0$ containing $\mathbf{Sob}$ (in particular, a Keimel-Lawson category). Is the product space of an arbitrary family of $\mathbf{K}$-determined spaces $\mathbf{K}$-determined?
\end{question}

\begin{question}\label{WD set infinite prod question} (\cite{xu-shen-xi-zhao1}) Let $X=\prod_{i\in I}X_i$ be the product space of a family $\{X_i: i\in I\}$ of $T_0$ spaces. If each $A_i\subseteq X_i (i\in I)$ is well-filtered determined, must the product set $\prod_{i\in I}A_i$ be well-filtered determined?
\end{question}

\begin{question}\label{WD space infinite prod question} (\cite{xu-shen-xi-zhao1}) Is the product space of an arbitrary family of well-filtered determined spaces well-filtered determined?
\end{question}

\section{Smyth power spaces}

The Smyth power spaces are very important structures in domain theory, which play a fundamental role in modeling the semantics of non-deterministic programming languages (see \cite{AJ94, redbook, Schalk}). There naturally arises a question of which topological properties are preserved by the Smyth power spaces.

It was proved by Schalk \cite{Schalk} that the Smyth power space $P_S(X)$ of a sober space $X$ is sober (see also \cite[Theorem 3.13]{Klause-Heckmann}), and the upper Vietoris topology (that is, the topology of Smyth power space) agrees with the Scott topology on $\mk (X)$ if $X$ is a locally compact sober space. Xi and Zhao \cite{xi-zhao-MSCS-well-filtered} proved that a $T_0$ space $X$ is well-filtered iff $P_S(X)$ is a $d$-space. Brecht and Kawai \cite{Brecht19} pointed out that $P_S(X)$ is second-countable for a second-countable $T_0$ space $X$. Recently, we further proved that a $T_0$ space $X$ is well-filtered iff $P_S(X)$ is well-filtered \cite{xuxizhao}.

\begin{theorem}\label{Heckman-Keimel theorem}\emph{(Heckmann-Keimel-Schalk Theorem)} (\cite{Klause-Heckmann, Schalk})  For a $T_0$ space $X$, the following conditions are equivalent:
\begin{enumerate}[\rm (1)]
            \item $X$ is sober.
 \item  For any $\mathcal A\in \ir(P_S(X))$ and $U\in \mathcal O(X)$, $\bigcap\mathcal K\subseteq U$ implies $K \subseteq U$ for some $K\in \mathcal A$.
           \item $P_S(X)$ is sober.
\end{enumerate}
\end{theorem}

For the well-filteredness of Smyth power space, we have the following similar result.

\begin{theorem}\label{Smythwf} (\cite{xu-shen-xi-zhao1, xuxizhao})
	For a $T_0$ space, the following conditions are equivalent:
\begin{enumerate}[\rm (1)]
		\item $X$ is well-filtered.
        \item $P_S(X)$ is a $d$-space.
        \item $P_S(X)$ is well-filtered.
\end{enumerate}
\end{theorem}

For the core compactness of Smyth power spaces, one has the following.

\begin{theorem}\label{Smyth power space core compact} (\cite{LJ20})
	For a $T_0$ space, the following conditions are equivalent:
\begin{enumerate}[\rm (1)]
		\item $X$ is locally compact.
        \item $P_S(X)$ is a $C$-space.
        \item $P_S(X)$ is locally compact.
        \item $P_S(X)$ is core compact.
\end{enumerate}
\end{theorem}

Hofmann and Lawson \cite{Hofmann-Lawson} (or \cite[Exercise V-5.25]{redbook}) constructed a second-countable core compact $T_0$ space $X$ in which every compact subset has empty interior (and hence $X$ is not locally compact), and hence by Theorem \ref{Smyth power space core compact}, the core compactness is not preserved by the Smyth power space functor.

\begin{theorem}\label{SmythWF} (\cite{xu-shen-xi-zhao1})  Let $X$ be a $T_0$ space. If $P_S(X)$ is well-filtered determined, then $X$ is well-filtered determined.
\end{theorem}

\begin{question}\label{Smyth WF question} (\cite{xu-shen-xi-zhao1}) Is the Smyth power space $P_S(X)$ of a well-filtered determined $T_0$ space $X$  well-filtered determined?
\end{question}

Now we consider the following question: for a first-countable (resp., second-countable) space $X$, does its Smyth power space $P_S(X)$ be first-countable (resp., second-countable)?

First, we have the following result, which was indicated in the proof of \cite[Proposition 6]{Brecht19}.

\begin{theorem}\label{Smyth CII} (\cite{Brecht19}) For a $T_0$ space, the following two conditions are equivalent:
\begin{enumerate}[\rm (1)]
\item $X$ is second-countable.
\item $P_S(X)$ is second-countable.
\end{enumerate}
\end{theorem}

\begin{theorem}\label{Smyth CII V=S} (\cite{Brecht19}) If $X$ is a second-countable sober space, then the upper Vietoris topology and the Scott topology on $\mk (X)$ coincide.
\end{theorem}

\begin{theorem}\label{min Compact countable is Smth CI} (\cite{XY20}) Let $X$ be a first-countable $T_0$ space. If $\mathrm{min}(K)$ is countable for any $K\in \mk (X)$, then
$P_S(X)$ is first-countable.
\end{theorem}

For the Alexandroff double circle $Y$ (see \cite[Example 3.1.26]{Engelking}), which is Hausdorff and first-countable, it is shown in \cite{XY20} that its Smyth power space $P_S(Y)$ is not first-countable.

For a general $T_0$ space $X$, Schalk \cite{Schalk} proved the following.

\begin{proposition}\label{LC sober V=S} (\cite{Schalk}) If $X$ is a locally compact sober space, then the upper Vietoris topology and the Scott topology on $\mk (X)$ coincide.
\end{proposition}

By Theorem \ref{core compact WF is sober} and Proposition \ref{LC sober V=S}, we have the following conclusion.

\begin{corollary}\label{LC WF domain V=S}  If $X$ is a core compact well-filtered space, then the upper Vietoris topology and the Scott topology on $\mk (X)$ coincide.
\end{corollary}

\begin{theorem}\label{Smyth CI S=V}  (\cite{XY20}) If $X$ is a well-filtered space and $P_S(X)$ is first-countable, then the upper Vietoris topology agrees with and the Scott topology on $\mk (X)$.
\end{theorem}

By Theorem \ref{min Compact countable is Smth CI} and Theorem \ref{Smyth CI S=V}, we get the following.

\begin{corollary}\label{Smyth compact count S=V}  (\cite{XY20}) If $X$ is a well-filtered space and $\mathrm{min}(K)$ is countable for any $K\in \mk (X)$, then the upper Vietoris topology agrees with and the Scott topology on $\mk (X)$.
\end{corollary}

So naturally one asks the following question.

\begin{question}\label{CI WF U V topol=Scott topol question} (\cite{XY20})  For a first-countable well-filtered space $X$, does the upper Vietoris topology and the Scott topology on $\mk (X)$ coincide?
\end{question}

A topological space $X$ is said to be \emph{consonant} if for every $\mathcal F\in \sigma (\mathcal O(X))$, there is a family $\{K_i : i\in I\}\subseteq \mk (X)$ with $\mathcal F=\bigcup_{K\in \mathcal A} \Phi (K)$ (see, e.g., \cite{Bou99, CW98, DGL95,
 Jean-2013, NS96}). The consonance is an important topological property (see \cite{AC99, Bou99, Bou96, DGL95, NS96}). In \cite{DGL95} it is proved that every C\'ech-complete space --- hence, in particular, every completely metrizable space --- is consonant. Bouziad \cite{Bou96} (see also \cite{CW98}) showed that the space of rationals with the subspace topology inherited from the reals is not
consonant. It was shown in \cite{dBSS16} that quasi-Polish spaces are consonant, and it is known that a
separable co-analytic metrizable space is consonant if and only if it is Polish (see \cite{Bou99}). Recently, in \cite{LJ20} it is shown that a $T_0$ space $X$ is locally compact iff $X$ is core compact and consonant.

For a sober space $X$, by the Hoffmann-Mislove Theorem (Theorem \ref{Hofmann-Mislove theorem}), we know that $X$ is consonant
iff the Scott topology on $\mathcal O(X)$ has a basis consisting of Scott-open filters.

\begin{theorem}\label{consonant campact-open=Isbell} (\cite{Jean-2013})
For a $T_0$ space $X$, the following statements are equivalent:
\begin{enumerate}[\rm (1)]
\item $X$ is consonant.
\item The compact open topology coincides with Isbell topology on
$\mathbf{Top}_o(X, S)$, where $S=\Sigma 2$ is the Sierpinski space.
\item The compact open topology coincides with Isbell topology on
$\mathbf{Top}_o(X, Y)$ for every $T_0$ space $Y$.
\end{enumerate}
\end{theorem}

Recently, Brecht \cite{Brecht19} proved that for a $T_0$ space $X$, the consonance of $X$ is equivalent to the commutativity of the upper and lower power spaces in the sense that $P_H(P_S(X))\cong P_S(P_H(X))$ under a naturally defined homeomorphism.

In \cite{Brecht19} Brecht and Kawai posed the following two questions.

\begin{question}\label{Smyth consonant question} (\cite{Brecht19}) For a consonant $T_0$ space $X$, is $P_S(X)$ also consonant?
\end{question}

\begin{question}\label{Hoare consonant question} (\cite{Brecht19}) For a consonant $T_0$ space $X$, is $P_H(X)$ also consonant?
\end{question}

In \cite{LJ20}, Lyu and Jia gave a partial answer to Question \ref{Smyth consonant question}.

\begin{proposition}\label{Smyth CI consonant} (\cite{LJ20}) If a $T_0$ space $X$ is consonant and $\Sigma \mathcal O(X)$ is first-countable, then $P_S(X)$ is consonant.
\end{proposition}

\section{$d$-reflections and well-filtered reflections of $T_0$ spaces}

\begin{definition}\label{WFtion} (\cite{Keimel-Lawson, xu20})
	Let $\mathbf{K}$ be a full subcategory of $\mathbf{Top}_0$ containing $\mathbf{Sob}$ and $X$ a $T_0$ space. A $\mathbf{K}$-\emph{reflection} of $X$ is a pair $\langle \widetilde{X}, \mu\rangle$ consisting of a $\mathbf{K}$-space $\widetilde{X}$ and a continuous mapping $\mu :X\longrightarrow \widetilde{X}$ satisfying that for any continuous mapping $f: X\longrightarrow Y$ to a $\mathbf{K}$-space, there exists a unique continuous mapping $f^* : \widetilde{X}\longrightarrow Y$ such that $f^*\circ\mu=f$, that is, the following diagram commutes.\\
\begin{equation*}
	\xymatrix{
		X \ar[dr]_-{f} \ar[r]^-{\mu}
		&\widetilde{X}\ar@{.>}[d]^-{f^*}\\
		&Y}
	\end{equation*}

\end{definition}

By a standard argument, $\mathbf{K}$-reflections, if they exist, are unique up to homeomorphism. We shall use $X^k$ to denote the space of the $\mathbf{K}$-reflection of $X$ if it exists. The space of $\mathbf{Sob}$-reflection of $X$ is the sobrification $X^s$ of $X$. The space of $\mathbf{Top}_d$-reflection (resp., $\mathbf{Top}_w$-reflection) of $X$ is denoted by $X^d$ (resp., $X^w$).

It is well-known that $\mathbf{Sob}$ is reflective in $\mathbf{Top}_0$ (see \cite{redbook, Jean-2013}). Using $d$-closures, Wyler \cite{Wyler} proved that $\mathbf{Top}_d$ is reflective in $\mathbf{Top}_0$. Later, Ershov \cite{Ershov_1999} showed that the $d$-completion (i.e., the $d$-reflection) of $X$ can be obtained by adding the closure of directed sets onto $X$ (and then repeating this process by transfinite induction). A more direct way to $d$-completions of $T_0$ spaces was given in \cite{ZhangLi}. In \cite{Keimel-Lawson}, using Wyler's method, Keimel and Lawson proved that for a full subcategory $\mathbf{K}$ of $\mathbf{Top}_0$ containing $\mathbf{Sob}$, if $\mathbf{K}$ has certain properties, then $\mathbf{K}$ is reflective in $\mathbf{Top}_0$. They showed that $\mathbf{Top}_d$ and some other categories have such properties.

For quite a long time, it is not known whether $\mathbf{Top}_w$ is reflective in $\mathbf{Top}_0$. Recently, following Keimel and Lawson's method, this problem is answered positively in \cite{wu-xi-xu-zhao-19}. Following Ershov's method, a more constructive well-filtered reflectors of $T_0$ spaces are presented in \cite{Shenchon}. In \cite{Liu-Li-Wu-2020}, following closely to Ershov's method in \cite{Ershov_2017}, another
way to construct a well-filterification of a $T_0$ space is given. In \cite {xu20}, for a full subcategory $\mathbf{K}$ of $\mathbf{Top}_0$ containing $\mathbf{Sob}$, the first author has provided a direct approach to $\mathbf{K}$-reflections of $T_0$ spaces.

\begin{theorem}\label{K-reflection}  (\cite{xu20})
	Let $\mathbf{K}$ be a full subcategory of $\mathbf{Top}_0$ containing $\mathbf{Sob}$ and Let $X$ a $T_0$ space. If $P_H(\mathbf{K}(X))$ is a $\mathbf{K}$-space, then the pair $\langle X^k=P_H(\mathbf{K}(X)), \eta_X^k\rangle$, where $\eta_X^k :X\longrightarrow X^k$, $x\mapsto\overline{\{x\}}$, is the $\mathbf{K}$-reflection of $X$.
\end{theorem}

\begin{definition}\label{K-adequate} (\cite{xu20}) Let $\mathbf{K}$ be a full subcategory of $\mathbf{Top}_0$ containing $\mathbf{Sob}$. $\mathbf{K}$ is called \emph{adequate} if for any $T_0$ space $X$, $P_H(\mathbf{K}(X))$ is a $\mathbf{K}$-space.
\end{definition}

\begin{corollary}\label{K-adequate reflective} (\cite{xu20})
	If $\mathbf{K}$ is adequate, then $\mathbf{K}$ is reflective in $\mathbf{Top}_0$.
\end{corollary}

\begin{theorem}\label{K-L reflective} (\cite{xu20}) Every Keimel-Lawson category is adequate, and hence reflective in $\mathbf{Top}_0$.
\end{theorem}

\begin{theorem}\label{d WF sober reflective} (\cite{xu20}) For $\mathbf{K}\in \{ \mathbf{Sob}, \mathbf{Top}_d, \mathbf{Top}_w\}$, $\mathbf{K}$ is adequate, and hence reflective in $\mathbf{Top}_0$.
	\end{theorem}

\begin{theorem}\label{K-prod} (\cite{xu20})
	Suppose that $\mathbf{K}$ is adequate and closed with respect to homeomorphisms. Then for any family $\{X_i:i\in I\}$ of $T_0$ spaces, the following two conditions are equivalent:
	\begin{enumerate}[\rm(1)]
		\item The product space $\prod_{i\in I}X_i$ is a $\mathbf{K}$-space.
		\item For each $i \in I$, $X_i$ is a $\mathbf{K}$-space.
	\end{enumerate}
\end{theorem}

\begin{corollary}\label{WF-prod} (\cite{Shenchon, xu-shen-xi-zhao1, xuxizhao})
	For any family $\{X_i:i\in I\}$ of $T_0$ spaces, the following two conditions are equivalent:
	\begin{enumerate}[\rm(1)]
		\item The product space $\prod_{i\in I}X_i$ is well-filtered.
		\item For each $i \in I$, $X_i$ is well-filtered.
	\end{enumerate}
\end{corollary}

\begin{theorem}\label{sobrification prod} (\cite{Hoffmann1, Hoffmann2})
	For any family $\{X_i: i\in I\}$ of $T_0$ spaces, $(\prod_{i\in I}X_i)^s=\prod_{i\in I}X_i^s$ (up to homeomorphism).
\end{theorem}

\begin{theorem}\label{K-reflectionprod}  (\cite{xu20})
	For an adequate $\mathbf{K}$ and a finite family $\{X_i: 1\leq i\leq n\}$ of $T_0$ spaces,  $(\prod\limits_{i=1}^{n}X_i)^k=\prod\limits_{i=1}^{n}X_i^k$  (up to homeomorphism).
\end{theorem}

\begin{corollary}\label{Keimel-Lawson reflectionprod}  (\cite{xu20})
	For a Keimel-Lawson category $\mathbf{K}$ and a finite family $\{X_i: 1\leq i\leq n\}$ of $T_0$ spaces,  $(\prod\limits_{i=1}^{n}X_i)^k=\prod\limits_{i=1}^{n}X_i^k$  (up to homeomorphism).
\end{corollary}

\begin{corollary}\label{d-reflectionprod} (\cite{Keimel-Lawson, xu20})
	For a finite family $\{X_i: 1\leq i\leq n\}$ of $T_0$ spaces, $(\prod\limits_{i=1}^{n}X_i)^d=\prod\limits_{i=1}^{n}X_i^d$  (up to homeomorphism).
\end{corollary}

\begin{corollary}\label{WF-reflectionprod} (\cite{Shenchon, xu20, xu-shen-xi-zhao1})
	For a finite family $\{X_i: 1\leq i\leq n\}$ of $T_0$ spaces, $(\prod\limits_{i=1}^{n}X_i)^w=\prod\limits_{i=1}^{n}X_i^w$  (up to homeomorphism).
\end{corollary}

\begin{question}\label{K-reflectionprod question} (\cite{xu20})
	Suppose that $\mathbf{K}$ is adequate and closed with respect to homeomorphisms. Does the $\mathbf{K}$-reflection functor preserves arbitrary products of $T_0$ spaces? Or equivalently, $(\prod\limits_{i\in I}X_i)^k=\prod\limits_{i\in I}X_i^k$ (up to homeomorphism) hold for any family $\{X_i : i\in I\}$ of $T_0$ spaces?
\end{question}

\begin{question}\label{K-L reflectionprod} (\cite{xu20})
	Let $\mathbf{K}$ be a Keimel-Lawson category. Does the $\mathbf{K}$-reflection functor preserves arbitrary products of $T_0$ spaces? Or equivalently,  $(\prod\limits_{i\in I}X_i)^k=\prod\limits_{i\in I}X_i^k$ (up to homeomorphism) hold for any family $\{X_i : i\in I\}$ of $T_0$ spaces?
\end{question}

\begin{question}\label{d-reflection infinite prod question} (\cite{xu20}) Does the $d$-reflection functor preserves arbitrary products of $T_0$ spaces? Or equivalently, does  $(\prod\limits_{i\in I}X_i)^d=\prod\limits_{i\in I}X_i^d$ (up to homeomorphism) hold for any family $\{X_i : i\in I\}$ of $T_0$ spaces?
\end{question}

\begin{question}\label{WF-reflection infinite prod question} (\cite{xu-shen-xi-zhao1}) Does the well-filtered reflection functor preserves arbitrary products of $T_0$ spaces? Or equivalently, does $(\prod\limits_{i\in I}X_i)^w=\prod\limits_{i\in I}X_i^w$ (up to homeomorphism) hold for any family $\{X_i : i\in I\}$ of $T_0$ spaces?
\end{question}

\begin{definition}\label{Smyth K} (\cite{xu20}) $\mathbf{K}$ is said to be a \emph{Smyth category}, if for any $\mathbf{K}$-space $X$, the Smyth power space $P_S(X)$ is a $\mathbf{K}$-space.
\end{definition}

By Theorem \ref{Heckman-Keimel theorem} and Theorem \ref{Smythwf}, $\mathbf{Sob}$ and $\mathbf{Top}_w$ are Smyth categories. Let $X$ be any $d$-space but not well-filtered (see Example \ref{dcpo is not well-filtered}). Then by Theorem \ref{Smythwf}, $P_S(X)$ is not a $d$-space. So $\mathbf{Top}_d$ is not a Smyth category.

\begin{theorem}\label{SmythK-D} (\cite{xu20})  Let $\mathbf{K}$ be an adequate Smyth category and $X$ a $T_0$ space $X$. If $P_S(X)$ is $\mathbf{K}$-determined, then $X$ is $\mathbf{K}$-determined.
\end{theorem}

\begin{question}\label{SmythWD} (\cite{xu20}) Let $\mathbf{K}$ be an adequate Smyth category. Is the Smyth power space $P_S(X)$ of a $\mathbf{K}$-determined $T_0$ space $X$ again $\mathbf{K}$-determined?
\end{question}

In \cite{ZhaoFan}, Zhao and Fan introduced \emph{bounded sobriety}, a
weak notion of sobriety, and showed that every $T_0$ space has a bounded sobrification.
Equivalently, $\mathbf{BSob}$, the category of all bounded sober spaces with continuous mappings, is reflective in $\mathbf{Top}_0$. By the way, using the method of \cite{xu20}, one can directly verify that $\mathbf{BSob}$ is adequate and hence reflective in $\mathbf{Top}_0$. In \cite{ZhaoHo}, Zhao and Ho introduced another weaker notion of sobriety --- the $k$-\emph{bounded sobriety}. A space $X$ is called $k$-\emph{bounded sober} if for any nonempty irreducible closed set $F$ whose supremum exists,
there is a unique point $x\in X$ such that $F=\cl\{x\}$. Zhao and Ho \cite{ZhaoHo} raised a question whether $\mathrm{KB}(X)$, the set of all irreducible closed sets of a $T_0$ space $X$ whose suprema exist, is the canonical $k$-bounded sobrification
of $X$ in the sense of Keimel and Lawson, with respect to the mapping $x\mapsto \cl\{x\}$.

Recently, Zhao Bin, Lu Jing and Wang Kaiyun \cite{ZJW19} gave a negative answer to Zhao and Ho's problem. Furthermore, it is proved in \cite{LWWZ20} that, unlike $\mathbf{Sob}$ and $\mathbf{BSob}$, the category $\mathbf{KBSob}$ of all $k$-bounded sober spaces with continuous mappings is not reflective in $\mathbf{Top}_0$.
 Using the counterexample in \cite{LWWZ20}, one can directly verify that $\mathbf{KBSob}$ is not adequate (cf. \cite[Conclusion section]{xu20}). Also it is noted that the $k$-sobriety is not preserved
by the Smyth power space functor (see \cite[Theorem 4.1]{LWWZ20}).

Let $\mathbf{Top}_k$ be the the category of all $T_0$
spaces with continuous mappings preserving all existing irreducible suprema. Note that every continuous mapping between $k$-bounded sober spaces preserves all existing irreducible suprema (see \cite[Theorem 4.5 and Lemma 5.3]{ZhaoHo}). In \cite{LWWZ20} the following question is stated.

\begin{question}\label{KBsob reflective question} (\cite{LWWZ20}) Is $\mathbf{KBSob}$ reflective in $\mathbf{Top}_k$?
\end{question}

\section{Strong $d$-spaces}

In order to uncover more finer links between $d$-spaces and well-filtered spaces, the notion of strong $d$-spaces has been introduced in \cite{xuzhao}.

\begin{definition}\label{strong $d$-space}  A $T_0$ space $X$ is called a strong $d$-space if for any $D\in \mathcal D(X)$, $x\in X$ and $U\in \mathcal O(X)$, $\bigcap_{d\in D}\ua d\cap \ua x\subseteq U$ implies $\ua d\cap \ua x\subseteq U$ for some $d\in D$. The category of all strong $d$-spaces with continuous mappings is denoted by $\mathbf{S}$-$\mathbf{Top}_d$.
\end{definition}

Figure 3 shows some relationships among classes of spaces lying between $d$-spaces and $T_2$ spaces (see \cite{xuzhao}).

\begin{figure}[ht]
	\centering
	\includegraphics[height=1.5in,width=4.0in]{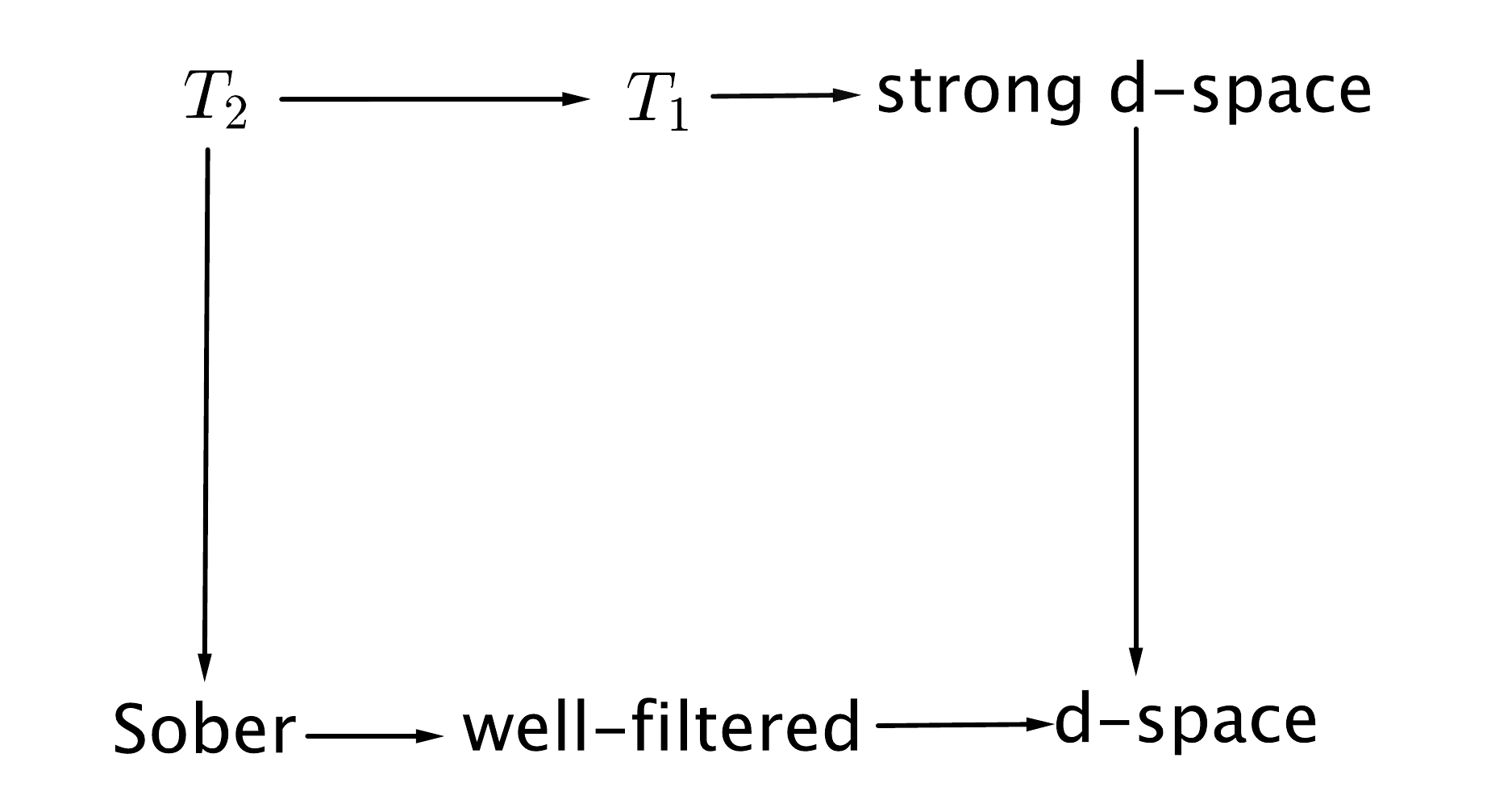}
	\caption{Relations of some spaces lying between $d$-spaces and $T_2$ spaces}
\end{figure}

The following two examples were given in \cite{xuzhao} (see also \cite{Shenchon, xu-shen-xi-zhao1}).

\begin{example}\label{examp1}
	Let $X$ be a countably infinite set and $X_{cof}$ the space of $X$  equipped with the \emph{co-finite topology} (the empty set and the complements of finite subsets of $X$ are open). Then $X_{cof}$ is $T_1$ and hence a strong $d$-space, but it is not well-filtered.
\end{example}

\begin{example}\label{examp3}
	Let $X$ be an uncountably infinite set and $X_{coc}$ the space equipped with \emph{the co-countable topology} (the empty set and the complements of countable subsets of $X$ are open). Then $X_{coc}$ is a well-filtered $T_1$ space and hence a strong $d$-space, but it not sober.
\end{example}

The following example shows that even for a continuous dcpo $P$ with $\Sigma~\!\!P$ coherent,  $\Sigma~\!\!P$ may not be a strong $d$-space.

\begin{example}\label{sober not strong d-space} (\cite{xuzhao}) Let $C=\{a_1, a_2, ..., a_n, ...\}\cup \{\omega_0\}$ and $P=C\cup\{b\}\cup\{\omega_1, ..., \omega_n, ...\}$ with the order (see Figure 4) generated  by
\begin{enumerate}[\rm (a)]
            \item $a_1<a_2<...<a_n<a_{n+1}<...$;
            \item  $a_n<\omega_0$ for all $n\in \mathbb{N}$;
            \item $b<\omega_n$ and $a_m<\omega_n$ for all $n, m\in \mathbb{N}$ with $m\leq n$.
\end{enumerate}

\begin{figure}[ht]
	\centering
	\includegraphics[height=2.2in,width=4.0in]{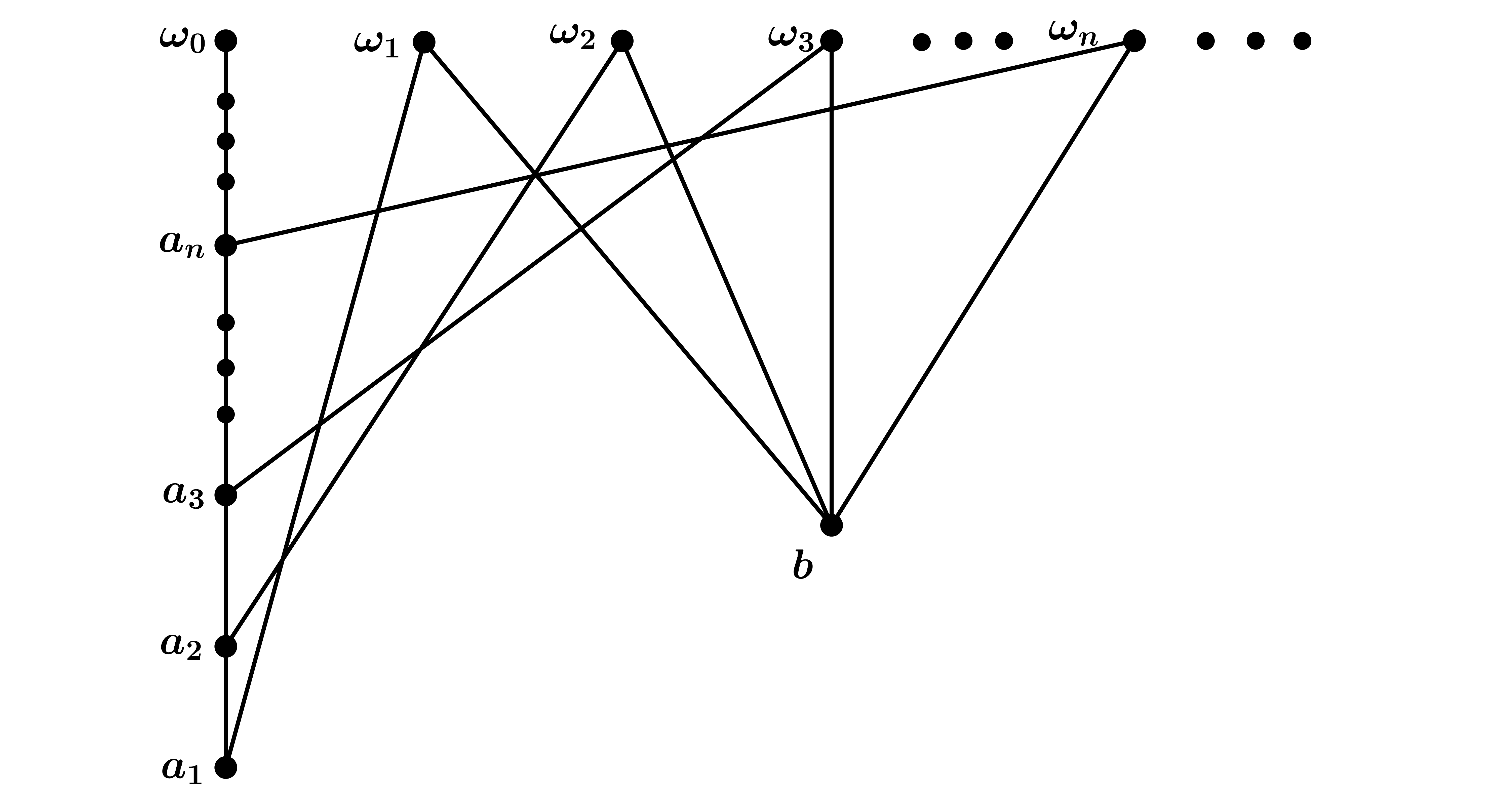}
	\caption{The Poset $P$}
\end{figure}

Then $x\ll x$ for all $x\in P\setminus \{\omega_0\}$. Therefore, $P$ is a continuous dcpo, and hence $\Sigma~\!\!P$ is sober. Clearly, $\ua a_1, \ua b\in \mk (\Sigma~)$, but $\ua a_1\cap \ua b=\{\omega_1, \omega_2, ..., \omega_n, ...\}$ is not Scott compact (note $\{\omega_n\}\in \sigma (P)$ for all $n\in \mathbb{N}$). Thus $\Sigma~\!\!P$ is not coherent. Clearly, $\bigcap_{n\in \mathbb{N}}\ua a_n\cap \ua b=\emptyset$, but $\ua a_n\cap \ua b=\{\omega_n, \omega_{n+1}, ...\}\neq\emptyset$, and consequently, neither $(P, \upsilon (P))$ nor $\Sigma ~\!\! P$ is a strong $d$-space.
\end{example}

\begin{proposition}\label{d-spacestrong d-space} (\cite{xuzhao})  If $X$ is a $d$-space and $\da (\ua x\cap A)\in \Gamma (X)$ for all $x\in X$ and $A\in \Gamma (X)$, then $X$ is a strong $d$-space.
\end{proposition}

\begin{lemma}\label{Scott compact closed} (\cite{xuzhao}) For a poset $P$ and $A\in \Gamma (\Sigma~\!\!P)$, the following two conditions are equivalent:
 \begin{enumerate}[\rm (1)]
 \item $\da (\ua x\cap A)\in \Gamma (\Sigma~\!\!P)$ for all $x\in P$.
 \item $\da (K\cap A)=\bigcup_{k\in K}\da (\ua k\cap A)\in \Gamma (\Sigma ~\!\!P)$ for all $K\in \mk (\Sigma ~\!\!P)$.
 \end{enumerate}
\end{lemma}

\begin{lemma}\label{Scott top-sd-space}  (\cite{xuzhao}) For a poset $P$, the following two conditions are equivalent:
 \begin{enumerate}[\rm (1)]
 \item $\Sigma~\!\!P$ is a strong $d$-space.
 \item $P$ is a dcpo, and for any $A\in \Gamma (\Sigma~\!\!P)$ and $x\in P$, $\da (\ua x\cap A)\in \Gamma(\Sigma~\!\!P)$.
 \end{enumerate}
 \end{lemma}

Theorem \ref{xi-lawsonWF}, Lemma \ref{Scott compact closed} and Lemma \ref {Scott top-sd-space} together imply the following.

\begin{corollary}\label{Scott top-sd WF} For a poset $P$, if $\Sigma P$ is a strong $d$-space, then it is well-filtered.
 \end{corollary}

It is easy to verify that if $A$ is a saturated subspace or a closed subspace of a strong $d$-space $X$, then $A$ is again a strong $d$-space.

\begin{proposition}\label{strong d-space retract} A retract of a strong $d$-space is a strong $d$-space.
\end{proposition}
\begin{proof}
	
It is well-known that a retract of a $T_0$ space is $T_0$ (cf. \cite{Engelking}). Assume that $X$ is a strong $d$-space and $Y$ a retract of $X$. Then there are continuous mappings $f:X\longrightarrow Y$ and $g:Y\longrightarrow X$ with $f\circ g={\rm id}_Y$. For any $D\in \mathcal D(Y)$, $y\in Y$ and $V\in \mathcal O(Y)$, if $\bigcap_{d\in D}\ua d\cap \ua y\subseteq V$, then $\bigcap_{d\in D}\ua g(d)\cap \ua g(y)\subseteq f^{-1}(V)$ and $\{g(d) : d\in D\}\in \mathcal D(X)$. Since $X$ is a strong $d$-space, there is $d\in D$ such that $\ua g(d)\bigcap \ua g(y)\subseteq f^{-1}(V)$, and hence $\ua d\bigcap \ua y=\ua f(\ua g(d)\bigcap \ua g(y))\subseteq \ua f(f^{-1}(V))\subseteq V$. Thus $X$ is a strong $d$-space.
\end{proof}

Let $\{X_i:i\in I\}$ be a family of $T_0$ spaces. Then for each $j\in I$, $X_j$ is a retract of the product space $\prod_{i\in I}X_i$. So if $\prod_{i\in I}X_i$ is a strong $d$-space, then by Proposition \ref{strong d-space retract}, any factor space $X_j$ ($j \in I$) is a strong $d$-space.

Concerning the other properties of strong $d$-spaces, we have the following three questions.

\begin{question}\label{strong d-space prod question} Is the product space of an arbitrary family of strong $d$-spaces again a strong $d$-space?
\end{question}

\begin{question}\label{strong d-space function question}
	Let $X$ be a $T_0$ space and $Y$ a strong $d$-space. Is the function space $\mathbf{Top}_0(X, Y)$ equipped with the pointwise convergence  topology a strong $d$-space? (cf. Proposition \ref{d-space function})
\end{question}

\begin{question}\label{super H-sober reflective question}
Is $\mathbf{S}$-$\mathbf{Top}_d$ reflective in $\mathbf{Top}_0$?
\end{question}

\section{Co-sober spaces}

In \cite{MJA_2004}, the concepts of $k$-\emph{irreducible} sets and \emph{co}-\emph{spaces} were introduced, and for a co-sober space, the alterative conditions for the dual Hofmann-Mislove Theorem were given (see \cite[Theorem 8.10]{MJA_2004}).

\begin{definition}\label{co-sober} (\cite{MJA_2004}) Let $X$ be a $T_0$ space and $G\in \mk (X)$.
 \begin{enumerate}[\rm (1)]
 \item $G$ is called $k$-\emph{irreducible} if it cannot be written as the union of two proper compact saturated subsets, or equivalently if $G=K_1\bigcup K_2$ for $K_1, K_2\in \mk (X)$ implies $G=K_1$ or $G=K_2$. The set of all $k$-irreducible compact saturated sets of $X$ is denoted by $\mathrm{K}$-$\ir(X)$.
 \item $X$ is called \emph{co}-\emph{sober} if for any  $K\in\mathrm{K}$-$\ir(X)$, there is a (unique) point $k\in X$ such that $K=\ua k$. The category of all co-sober spaces with continuous mappings is denoted by $\mathbf{Co}$-$\mathbf{Sob}$.
\end{enumerate}
\end{definition}

For co-sober spaces, Escard\'{o}, Lawson and Simpson \cite[Problem 9.7]{MJA_2004} asked wether every sober space is co-sober. In \cite{WX18}, Wen and the first author answered this problem in the negative by constructing a counterexample, and proved that saturated subspaces and closed subspaces of a co-sober space are co-sober.

\begin{theorem}\label{closure co-sober} (\cite{WX18}) Let $(X, \tau)$ be a co-sober space and $c : (X, \tau) \rightarrow (X, \tau)$ a continuous mapping. If $c : X \rightarrow X$ is a closure operator with respect to the specialization order of $X$, that is, $x\in \cl \{c(x)\}$ and $c(c(x))=c(x)$ for any $x\in X$. Then the subspace $(c(X), \tau|c(X))$ is a co-sober space.
\end{theorem}

\begin{theorem}\label{Smyth co-sober} (\cite{WX18}) Let $X$ be a $T_0$ space. If the Smyth power space $P_S(X)$ is co-sober, then $X$ is co-sober.
\end{theorem}

\begin{question}\label{Smyth co-sober question}  (\cite{WX18})
Is the Smyth power space $P_S(X)$ of a co-sober space $X$ again co-sober?
\end{question}

The following are more questions on co-sober spaces.

\begin{question}\label{kernel co-sober question} Let $(X, \tau)$ be a co-sober space and $k : (X, \tau) \rightarrow (X, \tau)$ a continuous mapping for which $k : X \rightarrow X$ is a kernel operator with respect to the specialization order of $X$, that is, $k(x)\in \cl \{x\}$ and $k(k(x))=k(x)$ for any $x\in X$. Is the subspace $(k(X), \tau|k(X))$ co-sober?
\end{question}

\begin{question}\label{co-sober retract question} Is a retract of a co-sober space co-sober?
\end{question}

\begin{question}\label{co-sober prod question} Is the product space of a family of co-sober spaces co-sober?
\end{question}

\begin{question}\label{co-sober reflectiv question}
Is $\mathbf{Co}$-$\mathbf{Sob}$ reflective in $\mathbf{Top}_0$?
\end{question}

\end{document}